\documentclass[12pt,a4paper, oneside]{amsart}

\usepackage{amsmath, nicefrac, verbatim, amsfonts, mathtools, amssymb, upgreek, xcolor, bbm}
\usepackage{graphics, xspace, enumerate}
\usepackage{stix}
\usepackage{dsfont}
\usepackage[a4paper,margin=2.5cm]{geometry}
\usepackage[bb=boondox]{mathalfa}

\usepackage{graphicx}
\usepackage[colorlinks=true,citecolor=red,urlcolor=blue,linkcolor=red,bookmarksopen=true,unicode=true,pdffitwindow=true]{hyperref}
\usepackage[english]{babel}
\usepackage[languagenames,fixlanguage]{babelbib}
\hypersetup{pdfauthor={}}
\hypersetup{pdftitle={}}

\usepackage{seqsplit,cleveref}

\theoremstyle{plain}
\newtheorem{theorem}{Theorem}[section]

\newtheorem{proposition}[theorem]{Proposition}
\theoremstyle{definition}

\newcommand {\Prob} {\ensuremath{\mathbb{P}}}
\newcommand {\R} {\ensuremath{\mathbb{R}}}
\newcommand {\ZZ} {\ensuremath{\mathbb{Z}}}
\newcommand {\N} {\ensuremath{\mathbb{N}}}

\newcommand{\process}[1]{\{#1(x,i;t)\}_{t\geq0}}

\newcommand{\df}{\coloneqq}

\newcommand{\E}{\mathrm{e}}

\newcommand{\PP}{\mathcal{P}}
\newcommand{\p}{\mathsf{p}}

\newcommand{\X}{\mathrm{X}}

\newcommand{\D}{\mathrm{d}}

\numberwithin{equation}{section}

\title[Non-central limit theorem for periodic regime-switching stable-like processes]{A note on functional  non-central limit theorem for periodic regime-switching stable-like processes}

\author[N.\ Sandri\'{c}]{Nikola Sandri\'{c}}
\address[Nikola\ Sandri\'{c}]{University of Zagreb Faculty of Science, Department of Mathematics\\ Zagreb\\Croatia}
\email{nikola.sandric@math.pmf.unizg.hr}

\makeatletter
\@namedef{subjclassname@2020}{%
	\textup{2020} Mathematics Subject Classification}
\makeatother


\begin{document}
\allowdisplaybreaks[4]

\begin{abstract} In this paper, we establish a functional non-central limit theorem for regime-switching stable-like processes with periodic coefficients. Our approach relies on stochastic analysis methods, in particular semimartingale characteristics, combined with the ergodic theory of Markov processes. If the set on which the order function achieves its minimum value $\underline{\alpha}$ has positive Lebesgue measure, then the limit is a rotationally invariant $\underline{\alpha}$-stable Lévy process. Otherwise, the limit is the zero process.
\end{abstract}

\maketitle

\section{Introduction}\label{S1}

In recent years, regime-switching stochastic processes have become increasingly popular due to their wide range of applications. They have been used as mathematical models in problems related to signal processing, biological and ecological systems, queuing networks, financial engineering, and other fields (see, e.g., \cite{Xi-Yin-Zhu-2019} and \cite{Yin-Zhu-Book-2010}). The main characteristic of these processes is the interaction between the dynamics of the system and discrete events that provide information about its current regime. This feature enables such systems to describe complex phenomena involving inherent uncertainty and environmental randomness.

These processes are typically described by systems of stochastic differential equations, together with the transition probability laws governing the switching mechanism. In this article, we focus on regime-switching stable-like processes. These are regime-switching processes whose spatial dynamics are driven by pure jumps governed by a heavy-tailed law depending on both the spatial state and the current regime, while the regime-switching mechanism is Markovian, taking values in a finite set of regimes with state-dependent jump rates. More precisely, for $m,d\in\N$, let $[m]:=\{1,\dots,m\}$,  $\alpha:\R^d\times [m]\to(0,2)$, $\gamma:\R^d\times[m]\to (0,\infty)$, and let
$\mathrm{q}_{ij}:\R^d\to\R$ be such that  $\mathrm{q}_{ij}(x)\ge0$ for  $i\neq j,$ where $i,j\in[m]$, and 
$$\mathrm{q}_{ii}(x)\,=\,-\sum_{j\neq i}\mathrm{q}_{ij}(x)\,,\quad i\in[m]\,.$$ If $m=1$, we set 
$\mathrm{q}_{11}\equiv0$. Throughout the article, we impose the following assumption:

\smallskip

\begin{description}
	\item[(A1)](Regularity) The functions $\alpha(x,i)$, $\gamma(x,i)$ and $\mathrm{q}_{ij}(x)$ are Lipschitz continuous in the spatial variable. Moreover, $$0\,<\,\underline{\alpha}\df\inf_{(x,i)\in\R^d\times[m]} \alpha(x,i) \,\le\, \sup_{(x,i)\in\R^d\times[m]} \alpha(x,i) \,<\,2$$ $$0\,<\,\inf_{(x,i)\in\R^d\times[m]} \gamma(x,i) \,\le\, \sup_{(x,i)\in\R^d\times[m]} \gamma(x,i) \,<\,\infty\,.$$ 
\end{description}

\noindent Furthermore, let $\{\mathrm{L}(t)\}_{t\ge0}$  be a $d$-dimensional rotationally invariant $\underline{\alpha}$-stable   L\'evy process starting at the origin, with L\'evy measure $\upnu_\mathrm{L}(\D y)=|y|^{-d-\underline{\alpha}}\D y$. Let ${\rm N}_\mathrm{L}(\D t,\D y)$ denote the corresponding Poisson random measure. Further, let ${\rm N}_\mathrm{q}(\D t,\D u)$ be a Poisson random measure on $[0,\infty)\times\R$ with intensity measure $\D t\times\D u$. Assume that ${\rm N}_\mathrm{L}(\D t,\D y)$ and ${\rm N}_\mathrm{q}(\D t,\D u)$ are defined on a stochastic basis $(\Omega, \mathcal{F}, \{\mathcal{F}_t\}_{t\ge0},\Prob)$ satisfying the usual conditions, and that they are independent.
Next, let  $c:\R^d\times[m]\times\R^d\to\R^d$ be defined as $$c(x,i,y)\,\df\,\left(\frac{\underline{\alpha}}{\alpha(x,i)}\gamma(x,i)\right)^{1/\alpha(x,i)}\lvert y\lvert^{\underline{\alpha}/\alpha(x,i)-1}y\,.$$ For $x\in\R^d$ and $i,j\in[m]$ with $i\neq j$, let $\Delta_{ij}(x)$ be the consecutive (with respect to the lexicographic ordering on $[m]\times[m]$) left-closed, right-open intervals of the real line, each having length $\mathrm{q}_{ij}(x)$. 
For a set $A$,  its indicator function
is denoted by $\mathbb{1}_{A}$, and the open ball of radius $\rho>0$ around $x\in\R^d$ is denoted by $\mathscr{B}_\rho(x)$.
Define $h:\R^d\times[m]\times\R\to\R$ by $$h(x,i,u)\,\df\, \sum_{j\in[m]}(j-i)\mathbb{1}_{\Delta_{ij}(x)}(u)\,.$$
According to \cite{Zhu-Yin-Baran-2015} and \cite[Proposition 2.1]{Xi-Yin-2011},  (A1)	implies that the regime-switching stochastic differential equation  \begin{align*}\D\X(x,i;t)&\,=\,\int_{\mathscr{B}_1(0)}c(\X(x,i;t-),\Lambda(x,i;t-),y)\,( {\rm N}_\mathrm{L}(\D t,\D y)-\D t\times\upnu_\mathrm{L}(\D  y))\\ &\ \ \ \  +\int_{\mathscr{B}^c_1(0)}c(\X(x,i;t-),\Lambda(x,i;t-),y)\,{\rm N}_\mathrm{L}(\D t,\D y)\\
\D\Lambda(x,i;t)&\,=\,\int_{\R}h(\X(x,i;t-),\Lambda(x,i;t-), u)\,{\rm N}_\mathrm{q}(\D t,\D u)\\
(\X(x,i;0),\Lambda(x,i;0))&\,=\,(x,i)\in\R^d\times[m] \end{align*} 
admits a unique nonexplosive strong solution $\{(\X(x,i;t),\Lambda(x,i;t))\}_{t\ge0}$ that is a 
strong Markov process with c\`{a}dl\`{a}g sample paths. Furthermore, 
let $\{\mathcal{L}_i\}_{i\in[m]}$ be a family of (non-local) operators of the form
$$
\mathcal{L}_if(x)\,\df\, \int_{\R^d}\left( f(x+ y)-f(x)-\mathbb{1}_{\mathscr{B}_1(0)}(y)\langle  y,\nabla f(x)\rangle\right)\frac{\gamma(x,i)}{|y|^{\alpha(x,i)+d}}\D y\,,
$$ acting on functions $f:\R^d\to\R$ for which the above integral is well-defined, and let $\mathcal{Q}=(\mathrm{q}_{ij})_{i,j\in[m]}:\R^d\to\R^{m\times m}$.  For $I:[m]\to\R$, define $$\mathcal{Q}(x)I(i)\,\df\,\sum_{j\in[m]}\mathrm{q}_{ij}(x)I(j)\,.$$
In the same references it is shown that for any   $$f\in\mathcal{D}\,\df\,\left\{g\in\mathcal{C}^2(\R^d\times[m]):(x,i)\mapsto\int_{\mathscr{B}^c_1(0)}|g(x+c(x,i,y))|\,\upnu_\mathrm{L}(\D y)\ \text{ is locally bounded}\right\}$$  it holds that the process \begin{equation}\label{eq:1}\left\{f(\X(x,i;t),\Lambda(x,i;t))-f(x,i)-\int_0^t\mathcal{L}f(\X(x,i;s-),\Lambda(x,i;s-))\,\D s\right\}_{t\ge0}\end{equation}
is a local martingale, where
$$\mathcal{L}f(x,i)\,\df\, \mathcal{L}_if(x,i)+\mathcal{Q}(x)f(x,i)\,.$$ Observe that  $\mathcal{C}_b^2(\R^d\times[m])\subseteq\mathcal{D}.$ This justifies referring to $\{(\X(x,i;t),\Lambda(x,i;t))\}_{t\ge0}$ as a regime-switching stable-like process.
Here, for $k\in\N_0\cup\{\infty\}$, $\mathcal{C}_b^k(\R^d)$ denotes the space of $k$-times continuously differentiable functions on $\R^d$ that are bounded and have bounded derivatives of all orders up to $k$. In particular, $\mathcal{C}^0(\R^d)$ and $\mathcal{C}_b^0(\R^d)$, abbreviated as $\mathcal{C}(\R^d)$ and $\mathcal{C}_b(\R^d)$, denote the spaces of  continuous functions and bounded continuous functions on $\R^d$, respectively. For functions on $\mathbb{R}^d\times[m]$, we use the convention that $f:\mathbb{R}^d\times[m]\to\mathbb{R}$ belongs to  $\mathcal{C}^k(\mathbb{R}^d\times[m])$ or $\mathcal{C}_b^k(\mathbb{R}^d\times[m])$ if $f(\cdot,i)$ belongs to the respective space for each $i\in[m]$.

Due to their inherent complexity, processes of this type are generally difficult to analyze. Properties such as the (strong) Feller property, transience, (positive) recurrence, and ergodicity have been studied in \cite{Zhu-Yin-Baran-2015,Chen-Chen-Tran-Yin-2019,Cloez-Hairer-2015, Kunwai-Zhu-2020, Lazic-Sandric-2022, Nguyen-Yin-2018, Pinsky-Scheutzow-1992, Shao-2015,Shao-2015-2, Shao-2015-3,Shao-Xi-2013, Tong-Majda-2016,Xi-2009, Xi-Yin-2011, Xi-Zhu-2017, Yin-Zhu-Book-2010}. The present article is concerned with the derivation of a tractable approximation of such a process. In the spirit of classical periodic homogenization theory, we impose spatial periodicity of the coefficients together with rapid oscillation on small scales, which enables us to derive a homogenized process that constitutes a valid approximation of the original one and is characterized by constant coefficients. To make this precise, we impose the following assumption:

\smallskip

\begin{description}
			\item[(A2)] (Periodicity)  For some $(\tau_1,\dots, \tau_{d})^T\in (0,\infty)^{d}$, for all $(k_1,\dots, k_{d})^T\in\ZZ^d$ and $(x,i)\in\R^d\times[m]$,  $$\alpha(x+(k_1\tau_1,\ldots,k_d\tau_d),i)\,=\,\alpha(x,i)\,,\qquad  \gamma(x+(k_1\tau_1,\ldots,k_d\tau_d),i)\,=\,\gamma(x,i)$$
		and 
		$$\mathrm{q}_{ij}(x+(k_1\tau_1,\ldots,k_d\tau_d))\,=\,\mathrm{q}_{ij}(x)\,.$$
\end{description}

 \noindent The main goal of this article is to prove a functional non-central limit theorem (under an appropriate scaling)  for the spatial component of  $\{(\X(x,i;t),\Lambda(x,i;t))\}_{t\ge0}$. Under (A1)-(A3), where (A3) is stated below, we show that if the set $$\bigcup_{i\in[m]}\{x\in\R^d:\alpha(x,i)=\underline{\alpha}\}$$ has positive Lebesgue measure, then the sequence of processes $\{  \X(x,i;nt)/n^{1/\underline{\alpha}}\}_{t\ge0}$, $n\in\N$,  converges 
in the Skorokhod space
 of  c\`{a}dl\`{a}g functions equipped with the  $J_1$-topology to a rotationally invariant $\underline{\alpha}$-stable L\'evy process starting from the origin. Otherwise, if the aforementioned set has zero Lebesgue measure, the sequence converges to the zero process (see \Cref{T3.1}).

\subsection{Literature review}  
Limit theorems for stochastic processes constitute a central topic in probability theory. Classical results in this area primarily concern processes with stationary and independent increments. In recent years, there has been increasing interest in extending these results to more complex models, including processes with spatial inhomogeneities and regime-switching dynamics. In this article, we establish a functional non-central limit theorem for the spatial component (which is, in general, non-Markovian) of a class of regime-switching stable-like processes with periodic coefficients. Our analysis combines stochastic analysis techniques based on semimartingale characteristics with ergodic arguments adapted to periodic switching environments.
Related results in the non-switching regime setting can be found in \cite{Franke-2006, Franke-2007}, while certain generalizations are given in \cite{Huang-Duan-Song-2024}. Analogous results in the discrete-time Markov chain setting are discussed in \cite{Sandric-JOTP-2014}. Results for the case of a constant order function are presented in \cite{Chen-Chen-Kumagai-Wang-2021, Franke-2007, Huang-Duan-Song-2022, Kassmann-Piatnitski-Zhizhina-2019}.
To the best of our knowledge, the only related results in the regime-switching setting are those concerning the case where the spatial component is a classical diffusion process. In \cite{Sandric-2024}, the case where the regime-switching mechanism does not depend on the spatial position is considered, while in \cite{Pahlajani-2025} these results are generalized to the case of a spatially dependent regime-switching mechanism.

\subsection{Notation} 
We summarize some  notation used throughout the article. 
For a topological space $(T,\mathcal{T})$,  we let  $\mathfrak{B}(T)$ denote the Borel $\sigma$-algebra on $T$, and $\mathcal{B}(T)$  the space of $\mathfrak{B}(T)/\mathfrak{B}(\R)$-measurable functions. 
For $f\in\mathcal{B}(\R^d)$ we let $\|f\|_\infty\df\sup_{x\in\R^d}|f(x)|$ denote the supremum norm, and 
define $\mathcal{B}_b(\R^d)\df\{f\in\mathcal{B}(\R^d)\colon \|f\|_\infty<\infty\}$. As above, we use the convention that $f:\mathbb{R}^d\times[m]\to\mathbb{R}$ belongs to  $\mathcal{B}(\mathbb{R}^d\times[m])$, $\mathcal{B}_b(\mathbb{R}^d\times[m])$
 if $f(\cdot,i)$ belongs to the respective space for each $i\in[m]$.
 The gradient of $f\in\mathcal{C}^1(\mathbb{R}^d\times[m])$ is denoted by $\nabla f(x,i)=(\partial_1 f(x,i),\dots,\partial_d f(x,i))^T$. For a Borel measure $\upmu$ on $\mathfrak{B}(\mathbb{R}^d\times[m])$ and $f\in\mathcal{B}(\mathbb{R}^d\times[m])$, we write $\upmu(f)=\int_{\mathbb{R}^d\times[m]} f(x,i)\,\upmu(\D x\times\{i\})$. For $\tau=(\tau_1,\dots,\tau_d)^T \in (0,\infty)^d$, define the lattice $\mathbb{Z}^d_\tau \df \{ (\tau_1 k_1, \dots, \tau_d k_d)^T \colon (k_1,\dots,k_d)^T\in \mathbb{Z}^d\}$. For $x\in\mathbb{R}^d$, set $[x]_\tau \df \{y\in\mathbb{R}^d \colon x-y \in \mathbb{Z}^d_\tau\}$. The torus $\mathbb{T}^d_\tau\times[m]$ is defined by identifying opposite faces of $[0,\tau]\times[m] = [0,\tau_1]\times\cdots\times[0,\tau_d]\times[m]$, i.e.,
$$
\mathbb{T}^{d}_\tau\times[m]\,\df\,\{([x]_\tau,i)\colon (x,i)\in\R^{d}\times[m]\}\,.$$
 The covering map $\Pi_\tau: \mathbb{R}^d\times[m]\to \mathbb{T}^d_\tau\times[m]$ is given by $\Pi_\tau(x,i) \df ([x]_\tau,i)$.
A function $f:\mathbb{R}^d\times[m]\to\mathbb{R}$ is called $\tau$-periodic if $f(x+k_\tau,i)=f(x,i)$ for all $(x,k_\tau,i)\in \mathbb{R}^d\times \mathbb{Z}^d_\tau \times[m]$. Every $\tau$-periodic function is uniquely determined by its restriction to $[0,\tau]\times[m]$, which defines $f_\tau:\mathbb{T}^d_\tau\times[m]\to\mathbb{R}$ via $f_\tau([x]_\tau,i) \df f(x,i)$. Using this identification,  we often write $x$ instead of $[x]_\tau$, and $f(x,i)$ instead of $f_\tau(x,i)$.

\subsection{Organization of the article}  In the next section, we discuss structural and ergodic properties of the regime-switching stable-like process ${(\X(x,i;t),\Lambda(x,i;t))}_{t\ge0}$.
Then, in \Cref{S3}, we prove that, under an appropriate scaling, the spatial component of this process satisfies a functional non-central limit theorem with a rotationally invariant $\underline{\alpha}$-stable limit, provided the set $$\bigcup_{i\in[m]}\{x\in\R^d:\alpha(x,i)=\underline{\alpha}\}$$
has positive Lebesgue measure. Otherwise, the process converges to the zero process. This result constitutes the main limit theorem of the article.

\section{Structural properties of the  regime-switching stable-like  process}\label{S2}

Let $${\rm p}(t,(x,i),\D y\times\{j\})\,:=\,\Prob((\X(x,i;t),\Lambda(x,i;t))\in\D y\times\{j\})$$ be the transition kernel  of $\{(\X(x,i;t),\Lambda(x,i;t))\}_{t\ge0}$. In addition to the properties mentioned in Section \ref{S1}, it is shown in  \cite{Zhu-Yin-Baran-2015} and \cite[Theorem 3.3]{Xi-Yin-2011} that (A1) implies that
\begin{itemize}
	\item[(i)] the process $\{(\X(x,i;t),\Lambda(x,i;t))\}_{t\ge0}$ satisfies the $\mathcal{C}_b$-Feller property, i.e., the corresponding semigroup of linear operators $\{\mathcal{P}_t\}_{t\ge0}$, defined by $$\mathcal{P}_tf(x,i)\,:=\,\int_{\R^d\times[m]}f(y,j)\,\p(t,(x,i),\D y\times\{j\})\,,\qquad f\in \mathcal{B}_b(\R^d\times[m])\,,$$
	satisfies  $\mathcal{P}_t(\mathcal{C}_b(\R^d\times[m]))\subseteq\mathcal{C}_b(\R^d\times[m])$ for every $t\ge0$.
	
		\smallskip

	\item[(ii)] the process in \eqref{eq:1} has the following representation
	\begin{equation}
	\begin{aligned}\label{eq:2}&\int_0^t\int_{\R^d}\bigl(f(\X(x,i;s-)+c(\X(x,i;s-),\Lambda(x,i;s-),y),\Lambda(x,i;s-))\\&\hspace{1.5cm} -f(\X(x,i;s-),\Lambda(x,i;s-))\bigr) ({\rm N}_\mathrm{L}(\D s,\D y)- \D s\times\upnu_\mathrm{L}(\D  y)) \\&+ \int_0^t\int_{\R}\bigl(f(\X(x,i;s-)+h(\X(x,i;s-),\Lambda(x,i;s-),u),\Lambda(x,i;s-))\\&\hspace{1.5cm} -f(\X(x,i;s-),\Lambda(x,i;s-))\bigr) ({\rm N}_\mathrm{q}(\D s,\D u)-\D s\times\D  u)\,.
	\end{aligned}
	\end{equation}
	\end{itemize}

 Next, observe that due to $\tau$-periodicity of the functions $\alpha(x,i)$, $\gamma(x,i)$, $\mathrm{q}_{ij}(x)$, $c(x,i,y)$ and $h(x,i,u)$ in the first variable, the processes $\{ (\X(x+ k_{\tau},i;t),\Lambda(x+k_\tau,i;t))\}_{t\ge0}$ and $\{ (\X(x,i;t)+ k_{\tau},$\linebreak $\Lambda(x,i;t))\}_{t\ge0}$   are indistinguishable. In particular, $$ \p(t,(x+ k_{\tau},i),B\times\{j\})\,=\, \p(t,(x,i),B- k_{\tau}\times\{j\})$$ for all  $B\in\mathfrak{B}(\R^d),$
which implies that the semigroup $\{\PP_t\}_{t\ge0}$ preserves the class of $ \tau$-periodic functions in $\mathcal{B}_b(\R^d\times[m])$. Thus,  according to \cite[Proposition 3.8.3]{Kolokoltsov-Book-2011}   the projection of   $\{( \X(x,i;t),$\linebreak $\Lambda(x,i;t))\}_{t\ge0}$ via $\Pi_{\tau}(x,i)$ onto  $\mathbb{T}^d_{\tau}\times[m]$, denoted by   $\{( \X^{\tau}(x,i;t),\Lambda(x,i;t))\}_{t\ge0}$, is a Markov process on $(\mathbb{T}^d_{\tau}\times[m],\mathfrak{B}(\mathbb{T}^d_{\tau}\times[m]))$ with transition kernel  
\begin{equation}\label{ES2.6}
 \p^{\tau}(t,(x,i),B\times\{j\})\,=\, \p\bigl(t,(z_x,i),\Pi_{\tau}^{-1}(B\times\{j\})\bigr)\end{equation} for  $(x,i)\in \mathbb{T}^d_{\tau}\times[m]$, $B\times\{j\} \in \mathfrak{B}(\mathbb{T}^d_{\tau}\times[m])$, and $(z_x,i) \in \Pi_{\tau}^{-1}(\{(x,i)\})$.
In particular, 
$\{(\X^{\tau}(x,i;t),\Lambda(x,i;t))\}_{t\ge0}$ is  a $\mathcal{C}_b$-Feller  process.
Together with    \cite[Theorem 3.1]{Meyn-Tweedie-AdvAP-II-1993}, this implies that $\{( \X^{\tau}(x,i;t),\Lambda(x,i;t))\}_{t\ge0}$ admits at least one invariant probability measure. Recall,  
 a Borel measure $\uppi$ on $\mathfrak{B}(\mathbb{T}^d_{\tau}\times[m])$  is said to be invariant for $\{(\X^{\tau}(x,i;t),\Lambda(x,i;t))\}_{t\ge0}$ if $$\int_{\mathbb{T}^d_{\tau}\times[m]} \p^{\tau}(t,(x,i),B\times\{j\})\,\uppi(\D x\times\{i\})\,=\,\uppi(B\times\{j\})$$
for $t\ge0$ and $B\times\{j\} \in \mathfrak{B}(\mathbb{T}^d_{\tau}\times[m])$.
Under the additional assumption (A3), stated below,  we  show that $\{( \X^{\tau}(x,i;t),\Lambda(x,i;t))\}_{t\ge0}$ 
admits a unique invariant probability measure. Moreover, the corresponding marginal distributions  converge as $t\to\infty$  to this  invariant measure in  total variation norm, and the convergence occurs at an exponential rate. We impose the following assumption:

\smallskip

\begin{description}
	\item[(A3)] (Irreducibility)  For any $i,j \in [m]$ with $i \neq j$, there exist $k \in \mathbb{N}$ and indices $i_0, \dots, i_k \in [m]$ satisfying $i_0 = i$, $i_k = j$, and $i_l \neq i_{l+1}$ for all $l = 0, \dots, k-1$, for which the set $$\{x\in \R^d:\mathrm{q}_{i_l,i_{l+1}}(x)>0\}$$ has positive Lebesgue measure for all $l = 0,\ldots,k-1$.
\end{description}	

\smallskip

\begin{proposition} \label{P2.1} Assume (A1)-(A3). Then, for all $t>0$ and $(x,i)\in\mathbb{T}_\tau^d\times[m]$, $$\p^\tau(t,(x,i),B\times\{j\})\,>\,0 $$ for all  $B\times\{j\}\in\mathfrak{B}(\mathbb{T}_\tau^d\times[m])$ satisfying $\mathrm{Leb}(\mathrm{pr}_{\R^d}\circ\Pi_{\tau}^{-1}(B\times\{j\}))>0.$ Here, $\mathrm{Leb}$ denotes the Lebesgue measure on $\mathfrak{B}(\R^d)$, and $\mathrm{pr}_{\R^d}:\R^d\times[m]\to\R^d$ is the projection mapping onto the first coordinate.
\end{proposition}

\begin{proof}
Clearly, it suffices to show that for all $t>0$,  $(x,i)\in\R^d\times[m]$ and $j\in[m]$, $$ \p\bigl(t,(x,i),B\times\{j\}\bigr)\,>\,0
$$ for all $B\in \mathfrak{B}(\R^d)$ with $\mathrm{Leb}(B)>0$. In particular, this implies that $\{(\X(x,i;t),\Lambda(x,i;t)\}_{t\ge0}$ and its skeleton Markov chain  $\{(\X(x,i;l),\Lambda(x,i;l)\}_{l\ge0}$ are irreducible  in the sense of \cite{Down-Meyn-Tweedie-1995}, with respect to the measure  $\upphi(B\times\{i\})\df\mathrm{Leb}(B)$.
 From
the proof of \cite[Theorem 4.8]{Xi-Yin-Zhu-2019}, we have that
\begin{equation*}
\begin{aligned}\label{eq:irr}
&\p(t,(x,i), B \times \{ j \} ) \,\ge\, 
\updelta_{ij}\, \Prob( \X^{(i)}(x;t)\in B)\, \E^{-\Theta t}\\ &+ \E^{-\Theta t}\sum_{k=1}^\infty \underset{0<t_1<\ldots<t_k<t}{\idotsint} \underset{\substack{i_0,\ldots,i_k \in [m] \\ i_l \neq i_{l+1}\\ i_0=i,\,  i_k=j}}{\sum} \int_{\R^d} \cdots \int_{\R^d} \Prob( \X^{(i_0)}(x;t_1)\in \D y_1)\, \mathrm{q}_{i_0i_1}(y_1) &\\
& \ \ \ \ \Prob( \X^{(i_1)}(y_1;t_2-t_1)\in \D y_2) \mathrm{q}_{i_1i_2}(y_2) \cdots \mathrm{q}_{i_{k-1}i_k}(y_k) \Prob( \X^{(i_k)}(y_k;t-t_k)\in B)\, \D t_1 \ldots \D t_k\,,
\end{aligned}
\end{equation*}
where $\updelta_{ij}$ denotes the Kronecker delta, and the constant $\Theta>0$ satisfies $$\inf_{(x,i)\in\R^d\times[m]}\mathrm{q}_{ii}(x)\,\ge\,-\Theta$$  (recall (A1) and (A2)). For each $i\in[m]$, the process $\{\X^{(i)}(x;t)\}_{t\ge0}$ is the unique, nonexplosive strong solution to 
\begin{align*}\D\X^{(i)}(x;t)&\,=\,\int_{\mathscr{B}_1(0)}c(X^{(i)}(x;t-),i,y)\,( {\rm N}_\mathrm{L}(\D t,\D y)-\D t\times\upnu_\mathrm{L}(\D  y))\\ &\ \ \ \  +\int_{\mathscr{B}^c_1(0)}c(X^{(i)}(x;t-),i,y)\,{\rm N}_\mathrm{L}(\D t,\D y)\\
\X^{(i)}(x;0)&\,=\,x\in\R^d\,. \end{align*} 
It is a
strong Markov process with c\`{a}dl\`{a}g sample paths, generated by $\mathcal{L}_i$ (see \cite[Proposition 2.1]{Xi-Yin-2011} ), and has transition kernel $${\rm p}^{(i)}(t,x,\D y)\,:=\,\Prob(\X^{(i)}(x;t)\in\D y)\,,$$ which satisfies ${\rm p}^{(i)}(t,x,B)>0$ for every $(t,x)\in(0,\infty)\times\R^d$ and $B \in\mathfrak{B}(\R^d)$ with $\mathrm{Leb}(B)>0$ (see \cite[Theorem 1.1 and Proposition 5.7]{Chen-Chen-Wang-2020}). 
In particular, for $i=j$, it follows that $\p(t,(x,i), B \times \{ j \} ) \,\ge\, 
 \Prob( \X^{(i)}(x;t)\in B)\, \E^{-\Theta t}$ for all $i\in[m]$, $(t,x)\in(0,\infty)\times\mathbb{R}^d$ and $B\in\mathfrak{B}(\mathbb{R}^d)$, and for $i\neq j$, there exist $k\in\mathbb{N}$ and $i_0,\dots,i_k\in[m]$ such that $i_0=i$, $i_k=j$, and $i_l\neq i_{l+1}$, and $$\mathrm{Leb}\bigl(\{x\in\mathbb{R}^d\colon\mathrm{q}_{i_l i_{l+1}}(x)>0\}\bigr)\,>\,0$$ for $l=0,\dots,k-1$, which in turn implies that $\mathrm{p}(t,(x,i),B\times\{j\})>0$ whenever $\mathrm{Leb}(B)>0$. The assertion of the proposition follows.
\end{proof}

\noindent From the previous proposition, it follows that for any  $O\in\mathfrak{B}(\mathbb{R}^d)$ with $\mathrm{Leb}(O)>0$, the processes 
 $\{( \X^{\tau}(x,i;t),$ $\Lambda(x,i;t))\}_{t\ge0}$  and    $\{(\X^\tau(x,i;l),\Lambda(x,i;l)\}_{l\ge0}$ are  irreducible with respect to the measure $$\uppsi(B\times\{i\})\,\df\, \mathrm{Leb}(\mathrm{pr}_{\R^d}\circ\Pi_{\tau}^{-1}(B\times\{i\})\cap O)\,,\qquad B\times\{i\}\in\mathfrak{B}(\mathbb{T}_\tau^d\times[m])\,.$$ Note that if $O$ is bounded, then $\uppsi$ is a finite measure. The irreducibility of $\{(\X^\tau(x,i;t),$ \linebreak$ \Lambda(x,i;t))\}_{t\ge0}$ immediately implies that it admits a unique invariant probability measure $\uppi$. Indeed, by \cite[Theorem 2.3]{Tweedie-1994}, every irreducible Markov process is either transient or recurrent. Recall, $\{ (\X^\tau(x,i;t), \Lambda(x,i;t)) \}_{t \ge 0}$ is said to be
 \begin{enumerate}
 	\item [(i)]
 	transient if it is  irreducible with respect to a $\sigma$-finite measure $\uppsi$ on $\mathfrak{B}(\mathbb{T}^d_{\tau}\times[m])$, and  there exists a countable
 	covering of $\mathbb{T}^d_{\tau}\times[m]$ with sets
 	$\{B_\ell\times\{j_\ell\}\}_{\ell\in\N}\subseteq\mathfrak{B}(\mathbb{T}^d_{\tau}\times[m])$, and for each
 	$\ell\in\N$ there exists a finite constant $c_\ell\ge0$ such that
 	$\int_0^{\infty}\p^{\tau}(t,(x,i),B_\ell\times\{j_\ell\})\,\D{t}\le c_\ell$ holds for all $(x,i)\in\mathbb{T}_\tau^d\times[m]$.
 	\item [(ii)]
 	recurrent if it is irreducible with respect to a $\sigma$-finite measure $\uppsi$ on $\mathfrak{B}(\mathbb{T}^d_{\tau}\times[m])$, and $\uppsi(B\times\{j\})>0$ implies that
 	$\int_{0}^{\infty}\p^{\tau}(t,(x,i),B\times\{j\})\,\D{t}=\infty$ for all $(x,i)\in\mathbb{T}_\tau^d\times[m]$.
 \end{enumerate}

\smallskip

 \noindent Since $\{ (\X^\tau(x,i;t), \Lambda(x,i;t)) \}_{t \ge 0}$ admits at least one invariant probability measure, it must be recurrent. Assume this is not the case, i.e.,
assume that $\{ (\X^\tau(x,i;t), \Lambda(x,i;t)) \}_{t \ge 0}$  is  transient 
 and that it admits a finite invariant measure $\uppi$. 
 Fix some $t>0$. Then, for each $\ell\in\N$, with $c_\ell$ and $B_\ell\times\{j_\ell\}$ as in (i) above, we have
 \begin{equation*}
 t\,\uppi(B_\ell\times\{j_\ell\})
 =\int_0^t\int_{\mathbb{T}^d_{\tau}\times[m]} \p^{\tau}(t,(x,i),B_\ell\times\{j_\ell\})\,\uppi(\D x\times\{i\})\,\D s
 \le c_\ell\uppi(\mathbb{T}^d_{\tau}\times[m]).
 \end{equation*}
 Letting  $t\to\infty$ yields a contradiction. By \cite[Theorem 2.6]{Tweedie-1994},  every recurrent Markov process admits a unique (up to constant multiples) invariant measure. Hence, $\{ (\X^\tau(x,i;t), \Lambda(x,i;t)) \}_{t \ge 0}$ admits a unique invariant probability measure.

\begin{proposition} \label{P2.2}Under (A1)-(A3),
	there exists constants $\gamma>0$ and $\varGamma>0$ such that $$\sup_{(x,i)\in\mathbb{T}_\tau^d\times[m]}\|\p^{\tau}(t,(x,i),\D y\times\{j\})-\uppi(\D y\times\{j\})\|_{\mathrm{TV}}\,\le\,\varGamma\E^{-\gamma t}\qquad \forall\,t\in[0,\infty)\,,$$ where $\lVert\cdot\rVert_{\mathrm{TV}}$ denotes the total variation norm on the space of
	signed measures on $\mathfrak{B}(\mathbb{T}_\tau^d\times[m])$.
\end{proposition}
\begin{proof} First, \cite[Theorems 5.1 and 7.1]{Tweedie-1994}, together with the $\mathcal{C}_b$-Feller property and irreducibility  of $\{( \X^{\tau}(x,i;t),\Lambda(x,i;t))\}_{t\ge0}$ with respect to $\uppsi$, it follows   that $\mathbb{T}_\tau^d\times[m]$ is a petite set for  the process (see \cite{Tweedie-1994} for  the definition of petite sets). Next, from \cite[Theorem 4.2]{Meyn-Tweedie-AdvAP-III-1993} (with $c=d=1$, $C=\mathbb{T}_\tau^d\times[m]$, $f(x,i)=V(x,i)\equiv1$, and $\mathcal{A}V(x,i)=\mathcal{L}V(x,i)\equiv0$), \Cref{P2.1} and \cite[Proposition 6.1]{Meyn-Tweedie-AdvAP-II-1993}, we see that $\{(\X^{\tau}(x,i;t),\Lambda(x,i;t))\}_{t\ge0}$ is aperiodic in the sense of \cite{Down-Meyn-Tweedie-1995}. The desired result then follows from \cite[Theorem 5.2]{Down-Meyn-Tweedie-1995} by taking $c=b=1$, $C=\mathbb{T}_\tau^d\times[m]$, $\tilde V(x,i)\equiv1$, and $\tilde{\mathcal{A}}\tilde V(x,i)=\mathcal{L}\tilde V(x,i)\equiv0$.
\end{proof}

\medskip

\noindent As a consequence of \Cref{P2.2},  we conclude  that for any $f\in\mathcal{B}_b(\mathbb{T}_\tau^d\times[m],\R)$ satisfying $\uppi(f)=0$, \begin{equation}\label{ES2.7}\sup_{(x,i)\in\mathbb{T}_\tau^d\times[m]}\|{\mathcal{P}}_t^{\tau}f(x,i)\|\,\le\,\varGamma\sup_{(x,i)\in\mathbb{T}_\tau^d\times[m]}\|f(x,i)\|\E^{-\gamma t} \qquad \forall\,  t\ge0\,,\end{equation}
where $\{\mathcal{P}_t^{\tau}\}_{t\ge0}$
denotes the semigroup of the process $\{( \X^{\tau}(x,i;t),\Lambda(x,i;t))\}_{t\ge0}$.


\section{Non-central limit theorem for the process $\process{\X}$}
\label{S3}

In this section, we prove the main result of this paper: a functional non-central limit theorem for the process $\process{\X}$.
Our approach is based on stochastic analysis methods, specifically the characteristics of semimartingales (see \cite{Jacod-Shiryaev-2003}). Let
$(\Omega,\mathcal{F},\{\mathcal{F}_t\}_{t\ge0},\mathbb{P},\{\mathrm{S}(t)\}_{t\ge0})$, which we denote simply  by
$\{\mathrm{S}(t)\}_{t\ge0}$ in the sequel, be a $d$-dimensional semimartingale. Let $h:\R^{d}\to\R^{d}$ be a truncation function, i.e., a
bounded Borel measurable function satisfying $h(x)=x$ in a neighborhood of
the origin.
Define
$$\check{\mathrm{S}}(h)(t)\,:=\,\sum_{s\leq t}(\Delta \mathrm{S}(s)-h(\Delta \mathrm{S}(s)))\quad
\textrm{and} \quad \mathrm{S}(h)(t)\,:=\,\mathrm{S}(t)-\check{\mathrm{S}}(h)(t)\,,\qquad t\geq0\,,$$ where the process
$\{\Delta \mathrm{S}(t)\}_{t\ge0}$ is defined as $\Delta \mathrm{S}(0):=\mathrm{S}(0)$ and $\Delta \mathrm{S}(t):=\mathrm{S}(t)-\mathrm{S}(t-)$ for $t>0$. The process $\{\mathrm{S}(h)(t)\}_{t\ge0}$ is a special
	semimartingale, i.e., it admits a unique decomposition
\begin{equation*}\label{SS}\mathrm{S}(h)(t)\,=\,\mathrm{S}(0)+M(h)(t)+B(h)(t)\,,\end{equation*} where $\{M(h)(t)\}_{t\ge0}$ is a local
martingale and $\{B(h)(t)\}_{t\ge0}$ is a predictable process of bounded
variation. Let $N(\omega,\D t,\D y)$ denote the
	compensator of the jump measure
	$$\upmu(\omega,\D t,\D y)\,:=\,\sum_{s:\,\Delta \mathrm{S}(s)(\omega)\neq 0}\updelta_{(s,\Delta \mathrm{S}(s)(\omega))}(\D t,\D y)\,,$$  and let $\{C(t)\}_{t\ge0}=\{(C^{kl}(t))_{1\leq k,l\leq d})\}_{t\geq0}$ be the quadratic co-variation
	process of $\{\mathrm{S}^{c}(t)\}_{t\ge0}$, the continuous martingale part of
	$\{\mathrm{S}(t)\}_{t\ge0}$. Here, $\updelta_p(\D p)$ denotes the Dirac measure 
	at point $p$.  The triplet $(B(h),C,N)$ is called
	the characteristics of the semimartingale $\{\mathrm{S}(t)\}_{t\ge0}$
	(relative to $h(x)$). Furthermore, let $\{\tilde{C}(h)(t)\}_{t\ge0}=\{(\tilde{C}(h)^{kl}(t))_{1\leq k,l\leq d})\}_{t\geq0}$ denote the quadratic covariation
	process of  $\{M(h)(t)\}_{t\ge0}$.  The triplet $(B(h),\tilde{C}(h),N)$ is called
	the modified characteristics of the semimartingale $\{\mathrm{S}(t)\}_{t\ge0}$
	(relative to $h(x)$). Equivalently, $(B(h),C,N)$ is 
	the characteristics of  $\{\mathrm{S}(t)\}_{t\ge0}$ relative to a truncation function $h(x)$ if, for every $f\in\mathcal{C}^2_b(\R^d)$, the process \begin{align*}&\Big\{f(\mathrm{S}(t))-f(\mathrm{S}(0))-\int_0^t\langle \nabla f(\mathrm{S}(s-),\D B(h)(s))\rangle-\frac{1}{2}\int_0^t\mathrm{Tr}(\nabla^2f(\mathrm{S}(s-))\D C^T(s))\\
	&-\int_0^t\int_{\R^d}\left(f(\mathrm{S}(s-)+y)-f(\mathrm{S}(s-))-\langle\nabla  f(\mathrm{S}(s-)), h(y)\rangle\right)\, N(\omega,\D s,\D y)\Big\}_{t\ge0}\end{align*} is a local martingale (see \cite[Theorem II.2.42]{Jacod-Shiryaev-2003}. Let $f\in\mathcal{C}^2_b(\R^d)$ be arbitrary, and let $\tilde{f}\in\mathcal{C}^2_b(\R^d\times[m])$ be defined as $\tilde{f}(x,i)\df f(x).$ As we have already commented, the process
	\begin{align*} &\tilde{f}(\X(x,i;t),\Lambda(x,i;t))-\tilde{f}(x,i)-\int_0^t\mathcal{L}\tilde{f}(\X(x,i;s-),\Lambda(x,i;s-))\,\D s\\
	&=\, f(\X(x,i;t))-f(x)-\int_0^t\mathcal{L}_{\Lambda(x,i;s-)}f(\X(x,i;s-))\,\D s\\
	&=\, f(\X(x,i;t))-f(x)-\int_0^t\int_{\R^d}\big( f(\X(x,i;s-)+ y)-f(\X(x,i;s-))\\ &\hspace{3.5cm}-\mathbb{1}_{\mathscr{B}_1(0)}(y)\langle  y,\nabla f(\X(x,i;s-))\rangle\big)\frac{\gamma(\X(x,i;s-),\Lambda(x,i;s-))}{|y|^{\alpha(\X(x,i;s-),\Lambda(x,i;s-))+d}}\,\D s\, \D y\,, \quad t\ge0\,,
\end{align*}
is a local martingale. From the preceding discussion, we see that the characteristics of\linebreak   $\process{\X}$  relative to $h(x)=x\mathbb{1}_{\mathscr{B}_1(0)}(x)$ are given by $\{B(h)(t)\}_{t\ge0}\equiv0$, $\{C(t)\}_{t\ge0}\equiv0$ and $$N(\omega,\D t,\D y)\,=\, \frac{\gamma(\X(x,i;t-)(\omega),\Lambda(x,i;t-)(\omega))}{|y|^{\alpha(\X(x,i;t-)(\omega),\Lambda(x,i;t-)(\omega))+d}}\,\D t\, \D y\,.$$
Furthermore, by \cite[Proposition II.2.17]{Jacod-Shiryaev-2003},  $$\tilde{C}(h)^{kl}(t)\,=\, \int_0^t\int_{\mathscr{B}_1(0)}y_ky_l \frac{\gamma(\X(x,i;s-),\Lambda(x,i;s-))}{|y|^{\alpha(\X(x,i;s-),\Lambda(x,i;s-))+d}}\,\D s\, \D y\,.$$

\noindent We are now in a position to state and prove the main result of this article.

\begin{theorem}\label{T3.1} Assume (A1)-(A3). Let $$A\,\df\,\bigcup_{i\in[m]}\{x\in\R^d:\alpha(x,i)=\underline{\alpha}\}\,.$$ If $\mathrm{Leb}(A)>0$, then
	
	$$\{  \X(x,i;nt)/n^{1/\underline{\alpha}}\}_{t\ge0}\,\xRightarrow[n\to\infty]{({\rm d})}\,\{\mathrm{Y}^{\underline{\alpha}}(0;t)\}_{t\ge0}\,,$$ where
	 $\{\mathrm{Y}^{\underline{\alpha}}(x;t)\}_{t\ge0}$ is a  $d$-dimensional rotationally invariant $\underline{\alpha}$-stable L\'evy process starting from $x\in\R^d$, with L\'evy measure  $\upnu_{\mathrm{Y}}(\D y)=\overline{\gamma} \lvert y\lvert^{-d-\underline{\alpha}} \D y,$   and $$\overline{\gamma}\,=\, \int_{\{(z,j)\in\mathbb{T}^d_{\tau}\times[m]:\alpha(z,j)=\underline{\alpha}\}} \gamma(x,i)\,\uppi(\D x\times\{i\})\,>\,0\,.$$ Here,   $\ \xRightarrow[]{({\rm d})}$ denotes  convergence in the space
	 of c\`{a}dl\`{a}g functions endowed with the Skorokhod $J_1$-topology (see \cite[Chapter VI]{Jacod-Shiryaev-2003}). If $A$ has zero Lebesgue measure and
	  $$\sup_{(x,i)\in\R^d\times[m]}\mathrm{q}_{ii}(x)\,<\,0\,,$$  then $\{  \X(x,i;nt)/n^{1/\underline{\alpha}}\}_{t\ge0}$ converges in the Skorokhod $J_1$-topology to the zero process.
	\end{theorem}

\begin{proof}  The semimartingale characteristics of the process $\{\X(x,i;nt)/n^{1/\underline{\alpha}}\}_{t\ge0}$
	with respect to $h(x)=x\mathbb{1}_{\mathscr{B}_1(0)}(x)$ are given by
	 \begin{align*}B_n(h)(t)&\,\equiv\,0\,,\\ N_n(\omega,\D t,\D y)&\,=\, n^{1-\alpha(\X(x,i;nt-)(\omega),\Lambda(x,i;nt-)(\omega))/\underline{\alpha}}\frac{\gamma(\X(x,i;nt-)(\omega),\Lambda(x,i;nt-)(\omega))}{|y|^{\alpha(\X(x,i;nt-)(\omega),\Lambda(x,i;nt-)(\omega))+d}}\,\D t\, \D y\,,
	 \\ \tilde{C}_n(h)^{kl}(t)&\,=\, \int_0^t\int_{\mathscr{B}_1(0)}y_ky_l\, n^{1-\alpha(\X(x,i;ns-),\Lambda(x,i;ns-))/\underline{\alpha}} \frac{\gamma(\X(x,i;ns-),\Lambda(x,i;ns-))}{|y|^{\alpha(\X(x,i;ns-),\Lambda(x,i;ns-))+d}}\,\D s\, \D y\,.
	 \end{align*}
If the set $A$ has positive Lebesgue measure, 	  then 
from  \Cref{P2.1} it follows that\linebreak  $\p^{\tau}(t,(x,i),\{(y,j)\in\mathbb{T}^d_{\tau}\times[m]:\alpha(y,j)=\underline{\alpha}\})>0$ for all $t>0$ and $(x,i)\in\mathbb{T}^d_{\tau}\times[m]$. Hence, for fixed $t>0$, we have that 
\begin{equation}
\begin{aligned}\label{inv}&\uppi(\{(y,j)\in\mathbb{T}^d_{\tau}\times[m]:\alpha(y,j)=\underline{\alpha}\})\\&\,=\,\int_{\mathbb{T}^d_{\tau}\times[m]} \p^{\tau}(t,(x,i),\{(y,j)\in\mathbb{T}^d_{\tau}\times[m]:\alpha(y,j)=\underline{\alpha}\})\,\uppi(\D x\times\{i\})\,>\,0\,,\end{aligned}
\end{equation}
 which ensures that $\overline{\gamma}>0$.
By \cite[Theorem VIII.2.17]{Jacod-Shiryaev-2003}, the first assertion of the theorem will follow if 
	$$\tilde{C}_n(h)^{kl}(t)\ \xrightarrow[n\to\infty]{\Prob}\ t\int_{\mathscr{B}_1(0)}y_ky_l \frac{\overline{\gamma}}{|y|^{\underline{\alpha}+d}}\,\D y\,,$$
	and 
	$$\int_0^t\int_{\R^d}g(y)\, N_n(\cdot,\D s,\D y)\ \xrightarrow[n\to\infty]{\Prob}\ t \int_{\R^d}g(y) \frac{\overline{\gamma}}{|y|^{\underline{\alpha}+d}}\,\D y\,,$$ for any
	$g\in\mathcal{C}_b(\R^d)$ that vanishes in a neighborhood of the origin. Here, $\xrightarrow[]{\Prob}$ denotes convergence in probability.
		We can decompose $\tilde{C}_n(h)^{kl}(t)$ as
		\begin{equation}\label{bo}
	\begin{aligned}
	&\tilde{C}_n(h)^{kl}(t)\\
	\,&=\, \int_0^t\int_{\mathscr{B}_1(0)}y_ky_l\, \mathbb{1}_{\{u:\alpha(\X(x,i;nu-),\Lambda(x,i;nu-))=\underline{\alpha}\}}(s)  \frac{\gamma(\X(x,i;ns-),\Lambda(x,i;ns-))}{|y|^{\underline{\alpha}+d}}\,\D s\, \D y\\
	&\ \ \ \ +\int_0^t\int_{\mathscr{B}_1(0)}y_ky_l\, \mathbb{1}_{\{u:\alpha(\X(x,i;nu-),\Lambda(x,i;nu-))>\underline{\alpha}\}}(s) n^{1-\alpha(\X(x,i;ns-),\Lambda(x,i;ns-))/\underline{\alpha}}\\
	&\hspace{2.6cm} \frac{\gamma(\X(x,i;ns-),\Lambda(x,i;ns-))}{|y|^{\alpha(\X(x,i;ns-),\Lambda(x,i;ns-))+d}}\,\D s\, \D y\,.
	\end{aligned}
	\end{equation}
	From \eqref{bo}, we first show that the  second term  converges to $0$ in $\mathrm{L}^2$. Define \begin{align*}
	U_n(z,i)\,\df\,\int_{\mathscr{B}_1(0)}y_ky_l\Big(&\mathbb{1}_{\{(v,j):\alpha(v,j)>\underline{\alpha}\}}(z,i) n^{1-\alpha(z,i)/\underline{\alpha}} \frac{\gamma(z,i)}{|y|^{\alpha(z,i)+d}}\\
	&-\int_{\mathbb{T}^d_{\tau}\times[m]} \mathbb{1}_{\{(v,j):\alpha(v,j)>\underline{\alpha}\}}(x,\ell) n^{1-\alpha(x,\ell)/\underline{\alpha}} \frac{\gamma(x,\ell)}{|y|^{\alpha(x,\ell)+d}}\uppi(\D x\times\{\ell\})\Big)\D y\,.
\end{align*}	
	Observe that $U_n(z,i)$ is $\tau$-periodic and  satisfies $$\int_{\mathbb{T}^d_{\tau}\times[m]}U_n(z,i)\,\uppi(\D z\times\{i\})\,=\,0\,.$$ 
	Using the Markov property, the integration-by-parts formula, and \eqref{ES2.7}, we obtain
	\begin{align*}&\mathbb{E}\left[\left(\int_0^{t}U_n(\X(x,i;ns),\Lambda(x,i;ns))\,\D s\right)^{2}\right]\\
	&\,=\, 2\int_0^{t}\int_0^{s}\mathbb{E}[U_n(\X(x,i;ns),\Lambda(x,i;ns))U_n(\X(x,i;nr),\Lambda(x,i;nr))]\,\D r\,\D s\\
	&\,=\,2\int_0^{t}\int_0^{s}\mathbb{E}[\mathbb{E}[U_n(\X(x,i;ns),\Lambda(x,i;ns))|\mathcal{F}_{nr}]U_n(\X(x,i;nr),\Lambda(x,i;nr))]\,\D r\,\D s \\&\,=\,2\int_0^{t}\int_0^{s}\mathbb{E}[\mathcal{P}^\tau_{n(s-r)}U_n(\X^\tau(x,i;nr),\Lambda(x,i;nr))U_n(\X^\tau(x,i;nr),\Lambda(x,i;nr))]\,\D r\,\D s\\
	&\,\le\,2\Gamma\max_{j\in[m]}\lVert U_n(\cdot,j)\rVert^2_\infty\int_0^{t}\int_0^{s}\E^{-n\gamma(s-r)}\,\D r\,\D s\\
	&\,\le\,\frac{2\Gamma t\max_{j\in[m]}\lVert U_n(\cdot,j)\rVert^2_\infty}{\gamma n}\,.\end{align*}
	Moreover, 
	$$|U_n(z,i)|\,\leq\,2\sup_{(z,i)\in\R^d\times[m]} \int_{\mathscr{B}_1(0)}|y|^2\frac{\gamma(z,i)}{|y|^{\alpha(z,i)+d}}\,\D y\,<\,\infty\,.$$
	Hence, by the triangle inequality and the definition of $U_n(z,i)$, we have
	\begin{align*}&\Bigg(\mathbb{E}\Bigg[\Bigg(
		\int_0^t\int_{\mathscr{B}_1(0)}y_ky_l\, \mathbb{1}_{\{u:\alpha(\X(x,i;nu-),\Lambda(x,i;nu-))>\underline{\alpha}\}}(s) n^{1-\alpha(\X(x,i;ns-),\Lambda(x,i;ns-))/\underline{\alpha}}\\
	&\hspace{2.6cm} \frac{\gamma(\X(x,i;ns-),\Lambda(x,i;ns-))}{|y|^{\alpha(\X(x,i;ns-),\Lambda(x,i;ns-))+d}}\,\D s\, \D y
		\Bigg)^{2}\Bigg]\Bigg)^{\frac{1}{2}}\\
		&\leq
		\left(\mathbb{E}\left[\left(\int_0^{t}U_n(\X(x,i;ns),\Lambda(x,i;ns))\,\D s\right)^{2}\right]\right)^{\frac{1}{2}}\nonumber\\
		&  
		+\left(\mathbb{E}\left[\left(\int_0^{t}\int_{\mathscr{B}_1(0)}y_ky_l\int_{\mathbb{T}^d_{\tau}\times[m]} \mathbb{1}_{\{(v,j):\alpha(v,j)>\underline{\alpha}\}}(x,\ell) n^{1-\alpha(x,\ell)/\underline{\alpha}} \frac{\gamma(x,\ell)}{|y|^{\alpha(x,\ell)+d}}\uppi(\D x\times\{\ell\})\,\D y\, \D s\right)^{2}\right]\right)^{\frac{1}{2}}\,.\end{align*}
	By the dominated convergence theorem, the second term  above   converges to
	zero as $n\to\infty$. Together with the bound on the first term, this shows that the second integral in \eqref{bo} converges to $0$ in 	$\mathrm{L}^2$, as required. We now show that the first term in \cref{bo} converges in $\mathrm{L}^2$ to $$t\int_{\mathscr{B}_1(0)}y_ky_l
	 \frac{\overline{\gamma}}{|y|^{\underline{\alpha}+d}}\, \D y\,.$$
	Define
	$$V(z,i)\,\df\,\int_{\mathscr{B}_1(0)}y_ky_l\left(1_{\{(v,j):\alpha(v,j)=\underline{\alpha}\}}(z,i)\frac{\gamma(z,i)}{|y|^{\underline{\alpha}+d}}-\frac{\overline{\gamma}}{|y|^{\underline{\alpha}+d}}\right)\,\D y\,.$$
	Clearly,  $V(z,i)$ is also $\tau$-periodic and satisfies $$\int_{\mathbb{T}^d_{\tau}\times[m]}V(z,i)\,\uppi(\D z\times\{i\})\,=0\,.$$  Proceeding as in the previous argument for $U_n(z,i)$ the $\mathrm{L}^2$-convergence of the first term in \eqref{bo} follows immediately.
Next, we show 	that $$\int_0^t\int_{\R^d}g(y)\, N_n(\cdot,\D s,\D y)\ \xrightarrow[n\to\infty]{\mathrm{L}^2}\ t \int_{\R^d}g(y) \frac{\overline{\gamma}}{|y|^{\underline{\alpha}+d}}\,\D y$$ for any
$g\in\mathcal{C}_b(\R^d)$ vanishing in a neighborhood of the origin. From \eqref{bo2}, we have
	\begin{equation}\label{bo2}
\begin{aligned}
&\int_0^t\int_{\R^d}g(y)\, N_n(\cdot,\D s,\D y)\\
\,&=\, \int_0^t\int_{\R^d}g(y)\, \mathbb{1}_{\{u:\alpha(\X(x,i;nu-),\Lambda(x,i;nu-))=\underline{\alpha}\}}(s)  \frac{\gamma(\X(x,i;ns-),\Lambda(x,i;ns-))}{|y|^{\underline{\alpha}+d}}\,\D s\, \D y\\
&\ \ \ \ +\int_0^t\int_{\R^d}g(y)\, \mathbb{1}_{\{u:\alpha(\X(x,i;nu-),\Lambda(x,i;nu-))>\underline{\alpha}\}}(s) n^{1-\alpha(\X(x,i;ns-),\Lambda(x,i;ns-))/\underline{\alpha}}\\
&\hspace{2.6cm} \frac{\gamma(\X(x,i;ns-),\Lambda(x,i;ns-))}{|y|^{\alpha(\X(x,i;ns-),\Lambda(x,i;ns-))+d}}\,\D s\, \D y\,.
\end{aligned}
\end{equation}
Define \begin{align*}
W_n(z,i)\,\df\,\int_{\R^d}g(y)\Big(&\mathbb{1}_{\{(v,j):\alpha(v,j)>\underline{\alpha}\}}(z,i) n^{1-\alpha(z,i)/\underline{\alpha}} \frac{\gamma(z,i)}{|y|^{\alpha(z,i)+d}}\\
&-\int_{\mathbb{T}^d_{\tau}\times[m]} \mathbb{1}_{\{(v,j):\alpha(v,j)>\underline{\alpha}\}}(x,\ell) n^{1-\alpha(x,\ell)/\underline{\alpha}} \frac{\gamma(x,\ell)}{|y|^{\alpha(x,\ell)+d}}\uppi(\D x\times\{\ell\})\Big)\D y\,.
\end{align*}	
Analogously to $U_n(z,i)$, we have
$$\mathbb{E}\left[\left(\int_0^{t}W_n(\X(x,i;ns),\Lambda(x,i;ns))\,\D s\right)^{2}\right]\,\le\, \frac{2\Gamma t\max_{j\in[m]}\lVert W_n(\cdot,j)\rVert^2_\infty}{\gamma n}\,,$$ and
$$|W_n(z,i)|\,\leq\,2\sup_{(z,i)\in\R^d\times[m]} \int_{\R^d}\lvert g(y)\lvert\frac{\gamma(z,i)}{|y|^{\alpha(z,i)+d}}\,\D y\,<\,\infty\,.$$
Hence, the second term in \cref{bo2} converges to $0$ in $\mathrm{L}^2$. Finally, to show that the first term converges in $\mathrm{L}^2$ to 
$$t\int_{\R^d}g(y)
\frac{\overline{\gamma}}{|y|^{\underline{\alpha}+d}}\, \D y\,,$$ 
define
$$Z(z,i)\,\df\,\int_{\R^d}g(y)\left(1_{\{(v,j):\alpha(v,j)=\underline{\alpha}\}}(z,i)\frac{\gamma(z,i)}{|y|^{\underline{\alpha}+d}}-\frac{\overline{\gamma}}{|y|^{\underline{\alpha}+d}}\right)\,\D y$$ and proceed exactly as above. This completes the proof of the first part of the theorem. 

Assume now that the set $A$ has zero Lebesgue measure, and that $$\sup_{(x,i)\in\R^d\times[m]}\mathrm{q}_{ii}(x)\,\le\,-\uptheta$$ for some $\uptheta>0$. From
the proof of \cite[Theorem 4.8]{Xi-Yin-Zhu-2019}, it follows  that
\begin{align*}
&\p(t,(x,i), B \times \{ j \} ) \,\le\, 
\updelta_{ij}\, \Prob( \X^{(i)}(x;t)\in B)\, \E^{-\uptheta t}\\ &+ \E^{-\uptheta t}\sum_{k=1}^\infty \underset{0<t_1<\ldots<t_k<t}{\idotsint} \underset{\substack{i_0,\ldots,i_k \in [m] \\ i_l \neq i_{l+1}\\ i_0=i,\,  i_k=j}}{\sum} \int_{\R^d} \cdots \int_{\R^d} \Prob( \X^{(i_0)}(x;t_1)\in \D y_1)\, \mathrm{q}_{i_0i_1}(y_1) &\\
& \ \ \ \ \Prob( \X^{(i_1)}(y_1;t_2-t_1)\in \D y_2) \mathrm{q}_{i_1i_2}(y_2) \cdots \mathrm{q}_{i_{k-1}i_k}(y_k) \Prob( \X^{(i_k)}(y_k;t-t_k)\in B)\, \D t_1 \ldots \D t_k\,,
\end{align*} where we use the same notation as in the proof of \Cref{P2.1}.
From this, together with \cite[Theorem 1.1 and Proposition 5.7]{Chen-Chen-Wang-2020} and \eqref{inv}, it follows that $$\uppi(\{(x,i)\in\mathbb{T}^d_{\tau}\times[m]:\alpha(x,i)=\underline{\alpha}\})\,=\,0\,.$$ Hence, $\overline{\gamma}=0$, and  by the same reasoning as in the previous case, we conclude that \linebreak  
$\{  \X(x,i;nt)/n^{1/\underline{\alpha}}\}_{t\ge0}$ converges in the Skorokhod $J_1$-topology to the zero process. \end{proof}

\smallskip

In the case $\mathrm{Leb}(A)=0$ — for instance, when the minimum of the order function is attained periodically at a single point — \Cref{T3.1} shows that the normalization $n^{1/\underline{\alpha}}$ is too strong: although the process does visit a neighbourhood of $A$, it does so with insufficient frequency to produce a non-degenerate limit under this scaling. If a non-trivial scaling limit exists at all, we expect that the appropriate normalization should reflect the rate at which $\alpha(x,i)$ approaches its minimum value $\underline{\alpha}$. We leave it as an open problem to determine whether a non-degenerate limit can be obtained under a different normalization, and if so, to identify it.

	\section*{Acknowledgements}
	Financial support through \textit{Croatian Science Foundation} under project No. 2277,  \textit{European Union – NextGenerationEU} through the
	National Recovery and Resilience Plan 2021-2026 Institutional grant of
	University of Zagreb Faculty of Science (IK IA 1.1.3. Impact4Math), and 
\textit{DIGIT.2.1.02.016} funded by the Digital, Innovation, and Green Technology Project – DIGIT Project (IBRD Loan No. 9558‑HR) are  gratefully acknowledged. I also thank the anonymous referee for the careful reading and thoughtful comments which have led to an improvement of the paper.

\section*{Disclosure statement} 
The author reports there are no competing interests to declare.

\bibliographystyle{abbrv}
\bibliography{References}

\end{document}